\documentclass[a4paper,11pt]{article}
\usepackage{amsfonts,amsmath,amssymb}
\usepackage{amsthm}
\usepackage{geometry}
\usepackage{graphicx}
\usepackage{mathrsfs}
\usepackage{tikz}
\usepackage{multicol}
\newtheorem{theorem}{Theorem}[section]
\newtheorem{lemma}{Lemma}[section]
\newtheorem{corollary}{Corollary}[section]
\newtheorem{conjecture}{Conjecture}[section]

\title{On the Locating Chromatic Number of Trees}

\author
{Yusuf Hafidh and Edy Tri Baskoro\\
	\\
	\normalsize{Combinatorial Mathematics Research Group}\\
	\normalsize{Faculty of Mathematics and Natural Sciences}\\
	\normalsize{Institut Teknologi Bandung}\\
	\normalsize{Jalan Ganesa 10 Bandung, Indonesia}\\
	\\
	\normalsize{Emails: yusufhafidh@math.itb.ac.id and ebaskoro@math.itb.ac.id}\\
}

\date{}

\begin{document} 
	
	\baselineskip16pt
	
	\maketitle 
	
\begin{abstract}
	Some coloring algorithms gives an upper bound for the locating chromatic number of trees with all the vertices not in an end-path colored by only two colors. That means, a better coloring algorithm could be achieved by optimizing the number of colors used in the end-paths. We provide an estimation of the locating chromatic number of trees using the locating chromatic number of its end-palms. We also study the locating chromatic number of palms, a subdivision of star. We also prove $\chi_L(S_n(k))=\Theta(n^{1/k})$; $\chi_L(S_n(3))=(1+o(1))\sqrt[3]{4n}$; and $\chi_L(O_n)=\left\lceil\log_3\left(\frac{n}{4}\right)\right\rceil+3$.
\end{abstract}


\section{Introduction}
Let $G=(V,E)$ be a simple connected graph. For any $u \in V$ and $S \subseteq V$, the distance from vertex $u$ to $S$ is defined by $d(u,S)=\min\{d(u,v) \mid v\in S\}$. A set of vertices $S$ {\em resolves} two vertices $u$ and $v$ if $d(u,S)\ne d(v,S)$. Let $c:V\to\{1,2,\cdots, k\}$ be a $k$-coloring of $G$ and $c^{-1}(i)=\{v\in V \mid c(v)=i\}$. A coloring $c$ is called a locating $k$-coloring (or simply a locating coloring) if for every two vertices, there exists a color class $c^{-1}(i)$ that resolves them. The \textit{locating chromatic number} of $G$, denoted by $\chi_L(G)$, is the smallest integer $k$ such that $G$ has a locating $k$-coloring.
The \textit{color code} of a vertex $v$ with respect to $c$ is given by $r_c(v)=\left(d(v,c^{-1}(1)),d(v,c^{-1}(2)),\cdots,d(v,c^{-1}(k))\right)$. The locating chromatic number is also called the metric chromatic number in \cite{MCN09}.

There has been some coloring algorithms that gives an upper bound for the locating chromatic number of trees, see \cite{ASB2019,DIDP,CHR02,DAM19}.
In almost all of these algorithms (\cite{DIDP}, \cite{CHR02}, and \cite{DAM19}), all the vertices which is not in an end-path (a path joining a leaf to its nearest branch) is colored by using only two colors.
One graph in particular having the locating chromatic number far from all the known upper bound is an olive $O_n$. From the algorithms in \cite{CHR02} and \cite{DAM19}, we have $\chi_L(O_n)\leq n+1$; and from \cite{DIDP}, we have $\chi_L(O_n)\leq \lceil\sqrt{n}\rceil+1$. The exact value for $\chi_L(O_n)$ is $\left\lceil\log_3\left(\frac{n}{4}\right)\right\rceil+3$ as stated in Theorem \ref{xlon}.

One of the reasons the known algorithms for the locating chromatic number is sill relatively far from the exact value is because those algorithms have not optimize the colors used in the end-palms (the union of all end-paths from a branch). That means, a better coloring algorithm could be achieved by optimizing the number of colors used in the palms. This motivates us to study the locating chromatic number of palms, which is a subdivision of a star. We believe that determining the locating chromatic number of all subdivision of stars will be a major part of determining the locating chromatic number of all trees.

In general, We will use the terminology in \cite{Book17}. A \textit{palm} $S_n(a_1,a_2,\cdots,a_n)$ for $n\geq2$, is the graph obtained from a star $S_{n}$ on $n+1$ vertices by subdividing the $i^{th}$ edge of $S_{n}$, $a_i-1$ times.
Formally, define the vertex-set and edge-set of $S_n(a_1,a_2,\cdots,a_n)$ as $V=\{a_{0}\}\cup\{a_{i,j}\mid 1\leq i\leq n, 1\leq j\leq a_i\}$ and $E=\{a_{0}a_{i,1}\mid 1\leq i \leq n\}\cup\{a_{i,j}a_{i,j+1}\mid 
1\leq i\leq n, 1\leq j\leq a_i-1\}$. The $k^{th}$ {\em  level} is the set of vertices of distance $k$ to the {\em hub} vertex $a_0$, and the $k^{th}$ {\em end-path} is the subgraph induced by the set $\{a_0\}\cup\{a_{k,j}:1\leq j\leq a_k\}$.
An \textit{olive} tree is defined as $O_n:=S_n(1,2,\cdots,n)$. Figure \ref{O5} is an example of an olive tree. When all the end-paths from a palm have the same length, we call it a {\em regular palm}, and it is denoted by $S_n(k):=S_n(k,k,\cdots,k)$.
\begin{figure}[h]
	\begin{center}
		\begin{tikzpicture}[scale=0.8]
		\begin{footnotesize}
		\draw[color=blue] (1,4)--(2,4) (0,3)--(3,3) (1,2)--(4,2) (1,1)--(5,1);
		\draw[color=blue] (1,4)--(0,3)--(1,2) (1,5)--(0,3)--(1,1);
		\draw[color=black,fill=blue!40] (0,3)circle(0.1) node[left]{$a_{0}$};
		\draw[color=black,fill=blue!40] (1,5)circle(0.1) node[above right]{$a_{1,1}$};
		\draw[color=black,fill=blue!40] (1,4)circle(0.1) node[above right]{$a_{2,1}$};
		\draw[color=black,fill=blue!40] (1,3)circle(0.1) node[above right]{$a_{3,1}$};
		\draw[color=black,fill=blue!40] (1,2)circle(0.1) node[above right]{$a_{4,1}$};
		\draw[color=black,fill=blue!40] (1,1)circle(0.1) node[above right]{$a_{5,1}$};
		\draw[color=black,fill=blue!40] (2,4)circle(0.1) node[above right]{$a_{2,2}$};
		\draw[color=black,fill=blue!40] (2,3)circle(0.1) node[above right]{$a_{3,2}$};
		\draw[color=black,fill=blue!40] (2,2)circle(0.1) node[above right]{$a_{4,2}$};
		\draw[color=black,fill=blue!40] (2,1)circle(0.1) node[above right]{$a_{5,2}$};
		\draw[color=black,fill=blue!40] (3,3)circle(0.1) node[above right]{$a_{3,3}$};
		\draw[color=black,fill=blue!40] (3,2)circle(0.1) node[above right]{$a_{4,3}$};
		\draw[color=black,fill=blue!40] (3,1)circle(0.1) node[above right]{$a_{5,3}$};
		\draw[color=black,fill=blue!40] (4,2)circle(0.1) node[above right]{$a_{4,4}$};
		\draw[color=black,fill=blue!40] (4,1)circle(0.1) node[above right]{$a_{5,4}$};
		\draw[color=black,fill=blue!40] (5,1)circle(0.1) node[above right]{$a_{5,5}$};
		\draw (-1,5.7)rectangle(6,0.5);
		\end{footnotesize}
		\end{tikzpicture}
		\caption{Graph $O_5=S_5(1,2,3,4,5)$}
		\label{O5}
	\end{center}
\end{figure}
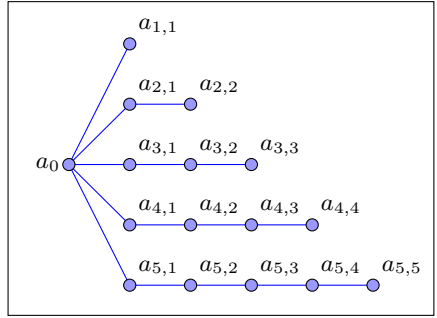

In the second section, we provide an estimation of the locating chromatic number of trees using the locating chromatic number of its end-palms.
In the third section we study the relation between the locating chromatic number of a graph and its maximum degree. The forth section discuss a tight upper and lower bound for the locating chromatic number of palms, we also prove that for every integer $k$ between the bounds, there is a palm having the locating chromatic number equal to $k$. 
In the last section, we took an asymptotic approach to study the locating chromatic number of regular palms, we prove that $\chi_L(S_n(k,k,\cdots,k))=\Theta(n^{1/k})$. This leads to the observation that $\chi_L(S_n(k,k,\cdots,k))$ is decreasing and goes to $\left\lceil\log_3\left(\frac{n}{4}\right)\right\rceil+3$ as a function of $k$, but it is increasing and unbounded as a function of $n$.

\section{Locating chromatic number of trees}

In this section, we provide an algorithm to make a locating coloring of any tree by utilizing the locating coloring of its palms. This algorithm requires we know a locating coloring of all of its end-palms, which we will study in the next sections. We also compare our algorithm with the algorithm given in \cite{DAM19} and a combined result in \cite{CHR02} and \cite{SLT75}.

\subsection{Coloring algorithm}

We introduce the notion of an end-path of a tree, that is, a path from a leaf to its nearest branch (a vertex with degree more than two). We call a branch, an end-branch, if it has at least two end-paths. Lastly, an end-palm is an end-branch together with all of its end-paths.

The following algorithm gives a locating coloring of a tree, provided we know a locating coloring of all of its end-palm.
\vspace{16pt}
\hrule
\vspace{3pt}
\centerline{Algorithm 1. Locating coloring of any tree.}
\vspace{3pt}
\hrule
\vspace{3pt}
\noindent\textbf{Input :} Tree $T$, locating coloring of all of its end-palms\\
\textbf{Output :} $c$, a coloring of $T$
\baselineskip 10pt
\begin{enumerate}
    \item Fix a vertex $w$ and for every vertex $u$, $c(u)= d(u,w) \pmod{2}$
    \item For every vertex $u$, $c(u)=c(u)+1$
\item $m\leftarrow 0$
\item For every end-palm of $T$, do:
\item \qquad Let $c'$ be a locating coloring of this end palm
\item \qquad Let $v$ be the branch vertex of the palm, and $v_1$ a neighbor of $v$ in the palm
\item \qquad $x\leftarrow c(v)$
\item \qquad $y\leftarrow 3-x$
\item \qquad Permute the colors in $c'$ such that $v$ is colored $x$ and $v_1$ is colored $y$
\item \qquad For every vertex $u$ in this palm, do:
\item \qquad \qquad If $c'(u)\leq 2$:
\item \qquad \qquad \qquad $c(u)\leftarrow c'(u)$
\item \qquad \qquad else
\item \qquad \qquad \qquad $c(u)\leftarrow m+ c'(u)$
\item \qquad $m\leftarrow m-2+\max\{c'(u): u\ \text{in this palm}\}$
\end{enumerate}
\hrule

\baselineskip 16pt

\begin{theorem}\label{xlT}
    Let $T$ be a tree with $b$ end-palms, $P_1,P_2,\cdots,P_b$, then
    $$\chi_L(T)\leq 2-2b+\sum_{i=1}^b\chi_L(P_i).$$
\end{theorem}
To prove Theorem \ref{xlT}, we need the following lemma
\begin{lemma}\cite{DIDP}\label{bridge}	
	Let $G$ be a graph and let $xy$ be a bridge of $G$. Let $G_x$ and $G_y$ be the component of $G-xy$ containing $x$ and $y$ respectively. Let $c$ be a coloring of $G$. If there exist $i$ and $j$ such that $c^{-1}(i)\subseteq V(G_x)$ and $c^{-1}(j)\subseteq V(G_y)$ then for any two vertices $u \in V(G_x)$ and $v \in V(G_y)$, their color codes are different.
\end{lemma}
\begin{proof}[\bf\em Proof of Theorem \ref{xlT}]
    Color $T$ using Algorithm $1$. Any two vertices in the same palm of $T$ is distinguished be the existence of the locating coloring in that palm, and any other two vertices is distinguished by lemma \ref{bridge}.
\end{proof}

Theorem \ref{xlT} makes us realize the importance of studying the locating chromatic number of palms. For the locating chromatic number of palms, see the next sections.

\subsection{Algorithm comparison}
    Let $T$ be a tree and $dim(T)$ its metric dimension, suppose $T$ has $l$ leaves and $\beta$ branches with at least one end-path. 
    One upper bounds of the locating-chromatic number of trees is given by combining the result in \cite{SLT75},
    \begin{align}
    \chi_L(T)\leq dim(T) + \chi(T),
    \end{align}
	and the result in \cite{CHR02},
    \begin{align}
    dim(T)=l-\beta
    \end{align}
    to get the following theorem
	\begin{theorem}\cite{SLT75,CHR02}\label{l-beta}
		Let $T$ be a tree with $l$ leaves and $\beta$ branch with at least one end-path, then $\chi_L(T)\leq l-\beta+2$.
    \end{theorem}
    Another upper bound is given in the following theorem.
	\begin{theorem}\label{l-b}\cite{DAM19} 
    	Let $T$ be a tree having $l$ leaves and $b$ branch with at least two end-paths, then $\chi_L(T)\leq l-b+2$.
	\end{theorem}
    
    The upper bound for the locating chromatic of trees in Theorem \ref{xlT} is better than the upper bounds in Theorems \ref{l-beta} and \ref{l-b}.
    \begin{theorem}
		Let $T$ be a tree having $l$ leaves, $\beta$ branch with at least one end-path, and $b$ branch with at least two end-paths. If $P_1,P_2,\cdots,P_b$ are the end-palms of $T$, then $$\chi_L(T)\leq2-2b+\sum_{i=1}^b\chi_L(P_i)\leq l-\beta+2\leq l-b+2.$$
    \end{theorem}
    \begin{proof}
        We will only prove the second inequality. Let $l_i$ be the number of leaves in $P_i$. There are $\beta-b$ end branch(es) with exactly one end-path, so $l=\beta-b+\sum_{i=1}^bl_i$.
        By Theorem \ref{batasxl}, $\chi_L(P_i)\leq l_i+1$ and the result follows.
    \end{proof}

\section{Maximum degree}

	In this section, we study the maximum degree of any graph with certain locating chromatic number. The correlation between the maximum degree of a graph with its metric dimension is used to characterize infinite graphs with finite metric dimension, see \cite{DAM12}. The maximum degree of a graph having certain locating chromatic number is also needed to characterize infinite graphs with finite locating chromatic number. In particular, we show that any graph with locating chromatic number $k\geq3$ must have maximum degree at most $4\cdot3^{k-3}$.
	
\begin{theorem}
\label{Delta xl}
If $G$ is a graph with $\chi_L(G)=k\geq3$, then $\Delta(G)\leq 4\cdot3^{k-3}$.
\end{theorem}

\begin{proof}
Let $G$ be a graph with $\chi_L(G)=k\geq3$ and $c:V(G)\to\{1,2,\cdots,k\}$ be a locating coloring of $G$.
Consider the color code of any vertex $v$, i.e., $r_c(v)=(a_1,a_2,\cdots,a_k)$. 
Without loss of generality, by permuting the colors, we may assume that $c(v)=1$, and so $a_1=0$, and $0<a_2\leq\cdots\leq a_k$.
Let $u$ be a neighbor of $v$, and $r_c(u)=(b_1,b_2,\cdots,b_k)$.
Then, $|a_i-b_i|\leq 1$ for all $i$ by the triangle inequality, and so $b_i \in \{a_i-1,a_i,a_i+1\}$ for all $i$.

Now we prove that $d(v)\leq 4\cdot3^{k-3}$ for any vertex $v$. To the contrary suppose that $d(v)\geq 4\cdot3^{k-3}+1$. 
First, group all the neighbors of $v$ depending to the distances to colors $4,5,\cdots,k-1,k$. All neighbors of $v$ with the same distances to colors $4,5,\cdots,k-1,k$ will be in the same group.
This means that their color codes of all members in a group will have the same ordinates in positions $4,5,\cdots,k-1,k$. 
Since the distance of any neighbor of $v$ to $c^{-1}(i)$ is either $a_i-1$, $a_i$, or, $a_i+1$, then 
there will be at most $3^{k-3}$ groups.
Since $v$ has $d(v)\geq 4\cdot3^{k-3}+1$ neighbors, by the pigeon hole principle there exists 
a group containing at least $5$ vertices, say $u_1,u_2,u_3,u_4,u_5$. The color codes of all the members of such a group will
be $(1,*,*,x_4,x_5, \cdots, x_k)$, for some fixed nonnegative integers $x_4,x_5,\cdots, x_k$.
		
If there exists a vertex $u$ in $U=\{u_1,u_2,\cdots,u_5\}$ with $c(u)\geq4$ then $c(u_i) = c(u)$ for 
all $i \in \{1,2,\cdots,5\}$.
Therefore, $0=a_1 < a_2\leq a_3 \leq \cdots \leq a_{c(u)}=1$, and so $a_2=a_3=1$. 
This implies that for every $u\in U$, $d(u,c_j)$ is 1 or 2, for $j=2,3.$
Since there are $5$ vertices in $U$ with $4$ possible representations, there will be two distinct vertices  with the same color code, a contradiction.
		
Now, the only possibility is that the color of each vertex $u \in U$ is either $2$ or $3$; it cannot be color $1$ because $u$ is adjacent to $v$ and $c(v)=1$. 
If all vertices in $U$ have the same color, say $c(u)=x$ for every $u \in U$ with $x=2$ or $x=3$, and 
let $y\in\{2,3\}-\{x\}$ (the other color), then we have $d(u,c^{-1}(1))=1$, $d(u,c^{-1}(x))=0$, and $d(u,c^{-1}(y))\in\{a_y-1,a_y,a_y+1\}$. 
This means that there are $5$ vertices in $U$ with $3$ possible representations, therefore 
there will be two vertices with the same color code, a contradiction.
		
So, $U$ must contain vertices of colors 2 and 3 only, and so $a_2=a_3=1$. 
Let $u \in U$. If $c(u)=2$ then $d(u,c^{-1}(1))=1$, $d(u,c^{-1}(2))=0$, and $d(u,c^{-1}(3))\in\{1,2\}$; and if $c(u)=3$ then $d(u,c^{-1}(1))=1$, $d(u,c^{-1}(2))\in\{1,2\}$, and $d(u,c^{-1}(3))=0$. 
Again, we have 4 possible representation for at least $5$ vertices, therefore there will be two vertices with the same color codes, a contradiction.
		
Therefore, $deg(v)\leq 4\cdot3^{k-3}$ for any vertex $v$. Thus, $\Delta(G)\leq 4\cdot3^{k-3}$.
\end{proof}


The tightness of this bound will be discussed in the next section. A different proof of Theorem \ref{Delta xl} was given in \cite{COR15}. In \cite{CHR02}, Chartrand et al. gave the following result.
\begin{theorem}
\label{ctr}
	(Theorem 4.3 in \cite{CHR02})
	Let $k\geq 3$. If $T$ is a tree for which $\Delta(T)>(k-1)2^{k-2}$, then $\chi_L(T)>k$.
\end{theorem}

In other form, we have that if $T$ is a tree with locating chromatic number $\chi_L(T)=k \;(\geq 3)$ then $\Delta(T) \leq (k-1)2^{k-2}$. This result is true only for $k=3$ and $k=4$. For $k \geq 5$, Theorem \ref{Delta xl} corrects the upper bound of the maximum degree of such tree $T$, namely $\Delta(T) \leq 4\cdot3^{k-3}$. Figure \ref{S_36(5)} gives a locating coloring with $k=5$ colors for a tree with $\Delta(T)=36$.

\begin{figure}[h]
\begin{center}
\begin{tikzpicture}[scale=0.45,rotate=270] 
	\begin{footnotesize}
	\foreach \x in {1,...,36} \draw[color=gray] (0,0)arc(180:90:2 and \x -18.5)--(12,\x -18.5);
    \draw[color=gray,fill=yellow!50] (0,0)circle(0.12)node[above]{\color{black}1};
	\foreach \x in {1,...,36}
	\draw[color=gray,fill=yellow!50]	(2,\x -18.5)circle(0.12) (4,\x -18.5)circle(0.12) (6,\x -18.5)circle(0.12) (8,\x -18.5)circle(0.12) (10,\x -18.5)circle(0.12) (12,\x -18.5) circle(0.12);
	\foreach \x in {1,...,9} \draw
	(2.1,4*\x-18.6)node[right]{2} (4.1,4*\x-18.6)node[right]{1}
	(2.1,4*\x-19.6)node[right]{3} (4.1,4*\x-19.6)node[right]{1}
	(2.1,4*\x-20.6)node[right]{2} (4.1,4*\x-20.6)node[right]{3}
	(2.1,4*\x-21.6)node[right]{3} (4.1,4*\x-21.6)node[right]{2};
	\foreach \x in {1}
	\draw (6.1,18.4-\x)node[right]{2} (8.1,18.4-\x)node[right]{1} (10.1,18.4-\x)node[right]{2} (12.1,18.4-\x)node[right]{1};
	\foreach \x in {2}
	\draw (6.1,18.4-\x)node[right]{3} (8.1,18.4-\x)node[right]{1} (10.1,18.4-\x)node[right]{3} (12.1,18.4-\x)node[right]{1};
	\foreach \x in {3}
	\draw (6.1,18.4-\x)node[right]{2} (8.1,18.4-\x)node[right]{3} (10.1,18.4-\x)node[right]{2} (12.1,18.4-\x)node[right]{3};
	\foreach \x in {4}
	\draw (6.1,18.4-\x)node[right]{3} (8.1,18.4-\x)node[right]{2} (10.1,18.4-\x)node[right]{3} (12.1,18.4-\x)node[right]{2};
	\foreach \x in {5}
	\draw (6.1,18.4-\x)node[right]{\color{red}4} (8.1,18.4-\x)node[right]{1} (10.1,18.4-\x)node[right]{2} (12.1,18.4-\x)node[right]{1};
	\foreach \x in {6}
	\draw (6.1,18.4-\x)node[right]{\color{red}4} (8.1,18.4-\x)node[right]{1} (10.1,18.4-\x)node[right]{3} (12.1,18.4-\x)node[right]{1};
	\foreach \x in {7}
	\draw (6.1,18.4-\x)node[right]{\color{red}4} (8.1,18.4-\x)node[right]{3} (10.1,18.4-\x)node[right]{2} (12.1,18.4-\x)node[right]{3};
	\foreach \x in {8}
	\draw (6.1,18.4-\x)node[right]{\color{red}4} (8.1,18.4-\x)node[right]{2} (10.1,18.4-\x)node[right]{3} (12.1,18.4-\x)node[right]{2};
	\foreach \x in {9}
	\draw (6.1,18.4-\x)node[right]{2} (8.1,18.4-\x)node[right]{\color{red}4} (10.1,18.4-\x)node[right]{2} (12.1,18.4-\x)node[right]{1};
	\foreach \x in {10}
	\draw (6.1,18.4-\x)node[right]{3} (8.1,18.4-\x)node[right]{\color{red}4} (10.1,18.4-\x)node[right]{3} (12.1,18.4-\x)node[right]{1};
	\foreach \x in {11}
	\draw (6.1,18.4-\x)node[right]{2} (8.1,18.4-\x)node[right]{\color{red}4} (10.1,18.4-\x)node[right]{2} (12.1,18.4-\x)node[right]{3};
	\foreach \x in {12}
	\draw (6.1,18.4-\x)node[right]{3} (8.1,18.4-\x)node[right]{\color{red}4} (10.1,18.4-\x)node[right]{3} (12.1,18.4-\x)node[right]{2};
	\foreach \x in {13}
	\draw (6.1,18.4-\x)node[right]{2} (8.1,18.4-\x)node[right]{1} (10.1,18.4-\x)node[right]{\color{blue}5} (12.1,18.4-\x)node[right]{1};
	\foreach \x in {14}
	\draw (6.1,18.4-\x)node[right]{3} (8.1,18.4-\x)node[right]{1} (10.1,18.4-\x)node[right]{\color{blue}5} (12.1,18.4-\x)node[right]{1};
	\foreach \x in {15}
	\draw (6.1,18.4-\x)node[right]{2} (8.1,18.4-\x)node[right]{3} (10.1,18.4-\x)node[right]{\color{blue}5} (12.1,18.4-\x)node[right]{3};
	\foreach \x in {16}
	\draw (6.1,18.4-\x)node[right]{3} (8.1,18.4-\x)node[right]{2} (10.1,18.4-\x)node[right]{\color{blue}5} (12.1,18.4-\x)node[right]{2};
	\foreach \x in {17}
	\draw (6.1,18.4-\x)node[right]{\color{red}4} (8.1,18.4-\x)node[right]{1} (10.1,18.4-\x)node[right]{\color{blue}5} (12.1,18.4-\x)node[right]{1};
	\foreach \x in {18}
	\draw (6.1,18.4-\x)node[right]{\color{red}4} (8.1,18.4-\x)node[right]{1} (10.1,18.4-\x)node[right]{\color{blue}5} (12.1,18.4-\x)node[right]{1};
	\foreach \x in {19}
	\draw (6.1,18.4-\x)node[right]{\color{red}4} (8.1,18.4-\x)node[right]{3} (10.1,18.4-\x)node[right]{\color{blue}5} (12.1,18.4-\x)node[right]{3};
	\foreach \x in {20}
	\draw (6.1,18.4-\x)node[right]{\color{red}4} (8.1,18.4-\x)node[right]{2} (10.1,18.4-\x)node[right]{\color{blue}5} (12.1,18.4-\x)node[right]{2};
	\foreach \x in {21}
	\draw (6.1,18.4-\x)node[right]{2} (8.1,18.4-\x)node[right]{\color{red}4} (10.1,18.4-\x)node[right]{\color{blue}5} (12.1,18.4-\x)node[right]{1};
	\foreach \x in {22}
	\draw (6.1,18.4-\x)node[right]{3} (8.1,18.4-\x)node[right]{\color{red}4} (10.1,18.4-\x)node[right]{\color{blue}5} (12.1,18.4-\x)node[right]{1};
	\foreach \x in {23}
	\draw (6.1,18.4-\x)node[right]{2} (8.1,18.4-\x)node[right]{\color{red}4} (10.1,18.4-\x)node[right]{\color{blue}5} (12.1,18.4-\x)node[right]{3};
	\foreach \x in {24}
	\draw (6.1,18.4-\x)node[right]{3} (8.1,18.4-\x)node[right]{\color{red}4} (10.1,18.4-\x)node[right]{\color{blue}5} (12.1,18.4-\x)node[right]{2};
	\foreach \x in {25}
	\draw (6.1,18.4-\x)node[right]{2} (8.1,18.4-\x)node[right]{1} (10.1,18.4-\x)node[right]{2} (12.1,18.4-\x)node[right]{\color{blue}5};
	\foreach \x in {26}
	\draw (6.1,18.4-\x)node[right]{3} (8.1,18.4-\x)node[right]{1} (10.1,18.4-\x)node[right]{3} (12.1,18.4-\x)node[right]{\color{blue}5};
	\foreach \x in {27}
	\draw (6.1,18.4-\x)node[right]{2} (8.1,18.4-\x)node[right]{3} (10.1,18.4-\x)node[right]{2} (12.1,18.4-\x)node[right]{\color{blue}5};
	\foreach \x in {28}
	\draw (6.1,18.4-\x)node[right]{3} (8.1,18.4-\x)node[right]{2} (10.1,18.4-\x)node[right]{3} (12.1,18.4-\x)node[right]{\color{blue}5};
	\foreach \x in {29}
	\draw (6.1,18.4-\x)node[right]{\color{red}4} (8.1,18.4-\x)node[right]{1} (10.1,18.4-\x)node[right]{2} (12.1,18.4-\x)node[right]{\color{blue}5};
	\foreach \x in {30}
	\draw (6.1,18.4-\x)node[right]{\color{red}4} (8.1,18.4-\x)node[right]{1} (10.1,18.4-\x)node[right]{3} (12.1,18.4-\x)node[right]{\color{blue}5};
	\foreach \x in {31}
	\draw (6.1,18.4-\x)node[right]{\color{red}4} (8.1,18.4-\x)node[right]{3} (10.1,18.4-\x)node[right]{2} (12.1,18.4-\x)node[right]{\color{blue}5};
	\foreach \x in {32}
	\draw (6.1,18.4-\x)node[right]{\color{red}4} (8.1,18.4-\x)node[right]{2} (10.1,18.4-\x)node[right]{3} (12.1,18.4-\x)node[right]{\color{blue}5};
	\foreach \x in {33}
	\draw (6.1,18.4-\x)node[right]{2} (8.1,18.4-\x)node[right]{\color{red}4} (10.1,18.4-\x)node[right]{2} (12.1,18.4-\x)node[right]{\color{blue}5};
	\foreach \x in {34}
	\draw (6.1,18.4-\x)node[right]{3} (8.1,18.4-\x)node[right]{\color{red}4} (10.1,18.4-\x)node[right]{3} (12.1,18.4-\x)node[right]{\color{blue}5};
	\foreach \x in {35}
	\draw (6.1,18.4-\x)node[right]{2} (8.1,18.4-\x)node[right]{\color{red}4} (10.1,18.4-\x)node[right]{2} (12.1,18.4-\x)node[right]{\color{blue}5};
	\foreach \x in {36}
	\draw (6.1,18.4-\x)node[right]{3} (8.1,18.4-\x)node[right]{\color{red}4} (10.1,18.4-\x)node[right]{3} (12.1,18.4-\x)node[right]{\color{blue}5};
	\end{footnotesize}
\end{tikzpicture}\\
		\caption{A tree $T$ with $\Delta(T)=36$ and $\chi_L(T)=5$.}
		\label{S_36(5)}
\end{center}
\end{figure}

\section{The locating chromatic number of palms}
	
In this section, we give a tight upper and lower bound for $\chi_L(S_n(a_1,a_2,\cdots,a_n))$. We show that the upper bounds of the maximum degree in Theorems \ref{Delta xl} are tight not only for general graph, but also for trees. We also show that for every integer $k\geq 3$ there is a palm with locating chromatic number $k$ and $\Delta=4\cdot 3^{k-3}$.

\begin{theorem}\label{batasxl}
	Let $n\geq 2$ and $G=S_n(a_1,a_2,\cdots,a_n)$ be a palm, then 
	\begin{align}
	\left\lceil\log_3\left(\frac{n}{4}\right)\right\rceil+3\leq \chi_L(G)\leq n+1.
	\end{align}
	Moreover, 
	\begin{align}
	\text{for every } k\ { with } \left\lceil\log_3\left(\frac{n}{4}\right)\right\rceil+3\leq k\leq n+1, \text{there exist a palm } G\ \text{with}\ \chi_L(G)=k.
	\end{align}
\end{theorem}
\begin{proof}
	We will prove the first part of Theorem \ref{batasxl}, the second part will be proven after Theorem \ref{xlon}. The lower bound is a direct consequence of Theorem \ref{Delta xl}. To prove the upper bound, let $c:V\to \{1,2,\cdots,n+1\}$ with $c(a_0)=n+1$, $c(a_{i,j})=i$ if $j$ is odd, and $c(a_{i,j})=n+1$ if $j$ is even. Any two vertices in the same end-path, say the $i^{th}$ end-path, is resolved by $c^{-1}(j)$ for every $j\ne i$, and any two vertices in different end-paths, say the $i^{th}$ and $j^{th}$ end-path, is resolved by $c^{-1}(i)$ and $c^{-1}(j)$. Thus, $c$ is a locating $(n+1)$-coloring of $G$, and the result follows.
\end{proof}

Next, we shall study the locating chromatic number of olive tree. This olive tree is a simpler counter example of Theorem 4.3 in \cite{CHR02} from the one given in \cite{COR15}.
The tightness of the upper bound of the maximum degree in Theorem \ref{Delta xl} and the lower bound in Theorem \ref{batasxl} is achieved by the olive tree. The upper bound in Theorem \ref{batasxl} is achieved by stars.

\vspace{12pt}
\hrule
\vspace{3pt}
\centerline{Algorithm 2. Coloring palms.}
\vspace{3pt}
\hrule
\vspace{3pt}
\noindent\textbf{Input :} Integers $n\geq 3$, $a_1,a_2,\cdots,a_n$.\\
\textbf{Output :} $c$, a coloring of $S_n(a_1,a_2,\cdots,a_n)$.
\begin{enumerate}
\item $k=\left\lceil\log_3\left(\frac{n}{4}\right)\right\rceil+3$.

\item For $i=1,2,\cdots,n$ write $i=4l+r$ with $r\in\{1,2,3,4\}$ and $l \in\mathbb{Z}$.
			
\item Write the number $l$ in the expression of $i=4l+r$ as a $(k-3)$-digit number in base $3$, allowing the first digit to be zero.

\item For two distinct integers $x$ and $y$, define an $(x,y)$-alternating sequence, as the sequence $\{x,y,x,y,\cdots\}$.
			
\item For $i=1,2,\cdots,n$, define an integer sequence $A_i=\{a^i_1,a^i_2,\cdots\}$ as follows.
\begin{enumerate}
  \item Write $i=4 l+r=4\times (l_{k}l_{k-1}\cdots l_5 l_4)_3 +r$; as in step 1.
				
  \item Initially, define $A_i$ for each $i$ as an $(x,y)$-alternating sequence with 
  $(x,y)=(2,1)$ if $r=1$, $(x,y)=(3,1)$ if $r=2$, $(x,y)=(2,3)$ if $r=3$, and $(x,y)=(3,2)$ if $r=4$.
\item For $t=4,5,\cdots, k$; if $l_t\ne0$, change the value of  $a^i_{2t+l_t-6}$ with $t$.
\end{enumerate}
		\item Assign $c(a_{0})=1$ and $c(a_{i,j})=a^i_j$ for $1 \leq j \leq a_i$, $i=1,2,\cdots,n$.
\end{enumerate}
\hrule
\vspace{6pt}
We give an example of Algorithm 2. If $n=108$, then $k=6$; write $i=57=4\times 14 + 1$ with $(14)_{10}=(112)_3$, $i=80=4\times 19 + 4$ with $(19)_{10}=(201)_3$, and $i=100=4\times 24 + 4$ with $(24)_{10}=(220)_3$, so $57=4\times(112)_3+1$, $80=4\times(201)_3+4$, and $100=4\times(220)_3+4$.

The sequences $A_i$ for $i=1,57,80,100$ are as follows. 
\begin{align*}
		A_{4\times (000)_3+1} = \{2,1,2,1,2,1,2,1,2,1,2,\cdots\}\qquad \qquad
		A_{4\times (112)_3+1} = \{2,1,2,4,5,1,6,1,2,1,2,\cdots\}\\
		A_{4\times (201)_3+4} = \{3,2,4,2,3,2,3,6,3,2,3,\cdots\}\qquad \qquad
		A_{4\times (220)_3+4} = \{3,2,3,2,3,5,3,6,3,2,3,\cdots\}
	\end{align*}

\begin{theorem}\label{xlon}
 For $n\geq 2$,  $\chi_L(O_n)=\left\lceil\log_3\left(\frac{n}{4}\right)\right\rceil+3$.
\end{theorem}
\begin{proof}

Since $\Delta(O_n)=n$, then by Theorem \ref{batasxl} we have $\chi_L(O_n) \geq \left\lceil\log_3\left(\frac{n}{4}\right)\right\rceil+3$.
Now, construct a locating coloring on $O_n$ with $k=\left\lceil\log_3\left(\frac{n}{4}\right)\right\rceil+3$ colors by using Algoritm 2. We will prove that the color codes of all vertices are different.

Note that vertex $a_{0}$ is the only vertex whose color $1$ and has neighbors with colors $2$ and $3$, so its color code is different from the color codes of the other vertices. 
Let $a_{i,j}$ and $a_{p,q}$ be two different vertices and write $i=4\times (\overline{i_{k}i_{k-1}\cdots i_5 i_4})+r_1$ and $p=4\times (\overline{p_{k}p_{k-1}\cdots p_5 p_4})+r_2$ 
as in 4(a). 
Let $(w,x)$ and $(y,z)$ be the alternating coloring for $A_i$ and $A_p$ in Step 4(b).
Consider the following cases. 
	
\textbf{Case I: $\mathbf{j\ne q}$.} 
In this case, $a_{i,j}$ and $a_{p,q}$ are in different level. 
Without loss of generality, let $j>q$. 
If $\{w,x\}=\{y,z\}$, then the color $s\in\{1,2,3\}-\{w,x\}$ is not used in $A_i$ and $A_p$. 
Since $j\ne q$, then $d(a_{i,j},c_s)\ne d(a_{p,q},c_s)$. 
If $\{w,x\}\ne\{y,z\}$, then there is a color used in $A_i$ but not used in $A_p$ and vice versa. 
Let $s\in \{y,z\}-\{w,x\}$; $s$ is the color used in $A_p$ but not in $A_i$. 
Note that either $a_{p,1}$ or $a_{p,2}$ is colored by $s$, so $d(a_{p,q},c_s)<q<j<d(a_{i,j},c_s)$.
	
\textbf{Case II: $\mathbf{j=q}$.} 
Since $a_{i,j}$ and $a_{p,q}$ are in the same level, they must be in different 
end-paths; $i\ne p$. 
If there is a $t\ (4\leq t\leq k)$ with $i_t\ne p_t$, then the position of vertex with color $t$ is different in $A_i$ and $A_p$, so $d(a_{i,j},c_t)\ne d(a_{p,q},c_t)$. 
If $i_t=p_t$ for all $t$, then $r_1\ne r_2$, which means that $A_i$ and $A_p$ have different alternating colorings. 
If $w\ne y$ then these two colors will distinguish $r_c(a_{i,j})$ and $r_c(a_{p,q})$ because they are in the same level; 
A similar argument can be applied if $x\ne z$.

Thus, we have constructed a locating coloring of $O_n$ with $k=\left\lceil\log_3\left(\frac{n}{4}\right)\right\rceil+3$ colors. Therefore, $\chi_L(O_n)=\left\lceil\log_3\left(\frac{n}{4}\right)\right\rceil+3$.
\end{proof}

	To prove the second part of Theorem $\ref{batasxl}$, we need to prove the following lemma
\begin{lemma}\label{1}
	Let $G=S_n(a_1,a_2,\cdots,a_n)$ be a palm and $G'=S_n(a_1,a_2,\cdots,a_i+1,\cdots,a_n)$ then \[\chi_L(G')\geq\chi_L(G)-1.\]
\end{lemma}
\begin{proof}
	Let $\chi_L(G')=p$ and $c'$ be a locating $p$-coloring of $G'$. We will construct a locating $(p+1)$-coloring of $G$, by doing so, we will have $\chi_L(G)\leq p+1$ and the result follows.
	
	Let $w$ be the only vertex in $G'$ but not in $G$, and $z$ the only neighbor of $w$. Define $c:V(G)\to \{1,2,\cdots,p+1\}$ with $c(z)=p+1$ and $c(v)=c'(v)$ if $v\ne z$. Without loss of generality let $c'(w)=1$. Suppose there are two different vertices $u$ and $v$ in $G$ with $r_c(u)=r_c(v)$, that means $d_G(u,c^{-1}(k))=d_G(v,c^{-1}(k))$ for $k=1,2,\cdots,p+1$.
	Since $d_G(u,c^{-1}(k))=d_{G'}(u,c^{-1}(k))$ for $k=2,3,\cdots,p$ and $r_{c'}(u)\ne r_{c'}(v)$, then $d_{G'}(u,c^{-1}(1))\ne d_{G'}(v,c^{-1}(1))$. 
	
	Without loss of generality, let $d_{G'}(u,c^{-1}(1))< d_{G'}(v,c^{-1}(1))$. Since $c'(w)=1$, we have $d_{G'}(u,c^{-1}(1))\leq d_{G'}(u,w)$, consider the following cases.\\
	{Case I : ${\ d_{G'}(u,c^{-1}(1))<d_{G'}(u,w)}$.}
	In this case, $d_{G}(u,c^{-1}(1))=d_{G'}(u,c^{-1}(1))$. This means, $d_{G'}(v,c^{-1}(1))\leq d_{G}(v,c^{-1}(1))=d_{G}(u,c^{-1}(1))=d_{G'}(u,c^{-1}(1))$, a contradiction.\\
	{Case II : ${\ d_{G'}(u,c^{-1}(1))=d_{G'}(u,w)}$.}
	In this case $d_{G'}(v,w)\geq d_{G'}(v,c^{-1}(1)) >d_{G'}(u,c^{-1}(1))=d_{G'}(u,w)$. This implies $d_G(v,c^{-1}(p+1))=d_G(v,z)=d_{G'}(v,w)-1>d_{G'}(u,w)-1=d_G(v,z)=d_G(v,c^{-1}(p+1))$, a contradiction.
	
	Thus $c$ is a locating $(p+1)$-coloring of $G$.
\end{proof}

Now we will prove the second part of Theorem \ref{batasxl}.
\begin{proof}[\bf\em Proof of Theorem \ref{batasxl} (2)]
	Define $$S_n=G_0\subseteq G_1 \subseteq \cdots \subseteq G_z=O_n=S_n(1,2,\cdots,n)$$ where $|G_{i+1}|=|G_{i}|+1$. Note that the previous sequence is not unique.	
	Consider the sequence $$n+1=\chi_L(G_0),\chi_L(G_1),\chi_L(G_2),\cdots,\chi_L(G_z)=\left\lceil\log_3\left(\frac{n}{4}\right)\right\rceil+3.$$
	
	From lemma \ref{1}, we have $\chi_L(G_{i+1})\geq \chi_L(G_i)-1$.
	Note that the previous sequence is decreasing in some of its terms because $n+1 \geq \left\lceil\log_3\left(\frac{n}{4}\right)\right\rceil+3$.
	Since each term is an integer and when it decrease, it can only decrease by $1$, the sequence will pass every integer between $\left\lceil\log_3\left(\frac{n}{4}\right)\right\rceil+3$ and $n+1$. This means for every integer $k$ between $\left\lceil\log_3\left(\frac{n}{4}\right)\right\rceil+3$ and $n+1$, there exist a palm $G_i$ with 
	$\chi_L(G_i)=k$. 
\end{proof}

By using Algorithm 2 and the arguments in the proof of Theorem \ref{xlon}, we get the following theorem. This also proves that the coloring in Figure \ref{S_36(5)} is a locating coloring.

\begin{theorem}
	Let $n\geq 2$ and $G=S_n(a_1,a_2,\cdots,a_n)$ be a palm with $1\leq a_1\leq a_2 \leq \cdots \leq a_n$. If $a_3\geq 2$, $a_{4\cdot3^k+1}\geq 2k+3$ and $a_{8\cdot 3^k+1}\geq 2k+4$ for all non negative integers $k\leq \log_3\left(\frac{n}{4}\right)$, then $\chi_L(G)=\left\lceil\log_3\left(\frac{n}{4}\right)\right\rceil+3$. $\square$
\end{theorem}
The following corollary is a special case for the previous theorem when $a_i=k$ for all $i$.
\begin{corollary}\label{>>}
	Let $n\geq 3$ and $k\geq 2 \left\lceil\log_3\left(\frac{n}{4}\right)\right\rceil+4$, then $\chi_L(S_n(k))=\left\lceil\log_3\left(\frac{n}{4}\right)\right\rceil+3$. $\square$
\end{corollary}

\begin{theorem}
	Let $G=S_n(a_1,a_2,\cdots,a_n)$ be a palm, $\chi_L(G)=n+1$ if and only if $G$ is a star. \label{xl=n+1}
\end{theorem}
\begin{proof}
	The $(\Leftarrow)$ part is trivial. For the $(\Rightarrow)$ part, let $G$ be a palm which is not a star. Without loss of generality, let $a_1>1$. Let $c:V\to \{1,2,\cdots,n\}$ with $c(a_0)=1$, $c(a_{1,j})=2$ if $j$ is odd, $c(a_{1,j})=3$ if $j$ is even, $c(a_{i,j})=i$ if $j$ is odd, and $c(a_{1,j})=1$ if $j$ is even. Its not hard to see that $c$ is a locating $n$-coloring of $G$ and the result follows.
\end{proof}

\section{Regular palms}

In this section, we will study the order of the locating chromatic number of regular palms.
If we fix $n$ and consider $\chi_L(S_n(k))$ as a function of $n$, we have $\chi_L(S_n(k)) \to \left\lceil\log_3\left(\frac{n}{4}\right)\right\rceil+3$ as $k\to\infty$ from Corollary $\ref{>>}$. However, if we fix $k$ and consider $\chi_L(S_n(k))$ as a function of $n$, $\chi_L(S_n(k))$ is increasing and unbounded (Lemma \ref{monoton}).

Let $f$ and $g$ be two function with integer variable $n$.
We say that $f=O(g)$ if there exist $c>0$ such that $|f(n)|\leq c\ |g(n)|$ for large values of $n$, $f=\Omega(g)$ if $g=O(f)$, and $f=\Theta(g)$ if $f=O(g)$ and $f=\Omega(g)$. We also denote $f=o(g)$ if $\lim \frac{f}{g}=0$. Note that $\lim \frac{f}{g}=1$ is equivalent to $f=(1+o(1))g$.

The order of $\chi_L(S_n(k))$ is given in the following theorem.

\begin{theorem}\label{orderxl}
	For every fixed positive integer $k$, $\chi_L(S_n(k))=\Theta\left(n^\frac{1}{k}\right)$.
\end{theorem}

We give the exact value of $\chi_L(S_n(k))$ for $k=1,2,3$ in the following Theorems.
\begin{theorem}\label{S12}
	For $n\geq 2$, $\chi_L(S_n(1))=n+1$ and $\chi_L(S_n(2))=\left\lceil\sqrt{n}\right\rceil+1$.
\end{theorem}

\begin{theorem}\label{S3}
	For integers $p\geq3$, let $f(p)=(p-1)\left\lfloor\frac{p^2}{4}\right\rfloor-\left\lfloor\frac{p^2-2p}{4}\right\rfloor$, then
	\begin{align}
	\chi_L(S_n(3))= p\ \Longleftrightarrow\ f(p-1)< n\leq f(p);\label{Sn3}
	\end{align}
	
	or simply $\chi_L\left(S_n(3)\right) =(1+o(1)) \sqrt[3]{4n}$.
\end{theorem}

We end this section with the following conjecture.

\begin{conjecture}
	For $k\geq 4$, $\chi_L(S_n(k)) =(1+o(1)) \left(\frac{k-1}{2}\right) \sqrt[k]{4n}$.
\end{conjecture}

\begin{center}
	{\bf Proof of the theorems \ref{orderxl}, \ref{S12}, and \ref{S3}}
\end{center}
To prove Theorem \ref{orderxl} we need to prove the following lemma.

\begin{lemma}\label{monoton}
	Let $k$ be a positive integer and $n\geq 3$, then $\chi_L(S_n(k))$ is increasing and unbounded as a function of $n$.
\end{lemma}
\begin{proof}
	The case for $k=1$ is clear. Let $k\geq 2$ and $n\geq 3$, we will prove that $\chi_L(S_{n}(k))\geq \chi_L(S_{n-1}(k))$.
	Let $\chi_L(S_{n}(k))=p$, then $p\leq n$ by Theorems \ref{batasxl} and \ref{xl=n+1}. Let $c$ be a locating $p$-coloring of $S_n(k)$ with $p$ as the color of the center. For $i=1,2,\cdots,p-1$; choose one vertex for color $i$ with the smallest distance to $a_0$, if there is more than one vertex, choose one arbitrary, this vertex is called the reference vertex of color $i$.
	Therefore we have $p-1$ chosen reference vertices. Since $p\leq n$, there exist an end-path not containing any reference vertex, if we remove this end-path, the remaining graph (with the remaining coloring) is a locating $p$-coloring of $S_{n-1}(k)$, because the color code for all vertices does not change. Therefore, $\chi_L(S_{n-1}(k))\leq p=\chi_L(S_{n}(k))$. The unbounded property comes from Theorem \ref{batasxl}.
\end{proof}

\begin{proof}[\bf\em Proof of Theorem \ref{orderxl}]
	We will prove that there exist $A,B>0$ such that
	\begin{align}\label{theta}
	A n^{\frac{1}{k}} \leq \chi_L(S_n(k)) \leq B n^{\frac{1}{k}}
	\end{align}
	for large values of $n$.
	
	First, we prove the second inequality in (\ref{theta}).
	Let $n\geq 4\cdot 3^{2k+1}$ and $\chi_L(S_n(k))=p+2\geq 2k+4$ from Theorem \ref{batasxl}. 
	For $i=1,2,\cdots,p$; let $A_i=\{a\in\{1,2,\cdots,p\} \mid a \equiv i\ \text{mod}\ k \}$.
	Let $m=\prod_{i=1}^k |A_i|$, we will construct a locating $(p+1)$-coloring $c$ of $S_m(k)$, as follows.
	\begin{enumerate}
		\item Let $c(a_0)=p+1$.
		\item Arrange the elements in $\mathbb{A}=A_1 \times A_2 \times \cdots \times A_k$ by their lexicographic order.
		\item For $i=1,2,\cdots,m$;
		let $\left(a_0,a_{i,1},a_{i,2},\cdots,a_{i,k}\right)$ be the $i^{th}$ end-path of $S_m(k)$, and\linebreak $(\alpha_{i1},\alpha_{i2},\cdots,\alpha_{ik})$ be the $i^{th}$ element in $\mathbb{A}$, then define $c(a_{i,j})=\alpha_{ij}$.
	\end{enumerate}
	It is easy to verify that the previous coloring is a locating $(p+1)$-coloring of $S_m(k)$, thus $\chi_L(S_m(k))\leq p+1<\chi_L(S_n(k))$. By Lemma \ref{monoton}, we have  $n>m$, therefore
	\begin{align*}
	n> \prod_{i=1}^k |A_i|\geq \prod_{i=1}^{k} \left\lfloor\frac{p}{k}\right\rfloor \geq \left(\frac{p}{k}-1\right)^k \geq \left(\frac{p+2}{2k}\right)^k,
	\end{align*}
	which is equivalent to $p+2\leq (2k) n^{\frac{1}{k}}$. Thus, the bound in (\ref{theta}) is satisfied for $n\geq 4\cdot 3^{2k+1}$, and $B=2k$.
	
	Now, we only need to prove the first inequality in (\ref{theta}). Let $q=\chi_L(S_n(k))$. There are $k$ vertices in an end-path without the center of $S_n(k)$. In a resolving $q$-coloring, each vertex has $q$ possible color, therefore there are at most $q^k$ possible ways to color an end-path. Since two different end-paths cannot have the same coloring, then there are at most $q^k$ end-paths, thus $n\leq q^k$ which is equivalent to $n^{\frac{1}{k}}\leq q=\chi_L(S_n(k))$. 
	Hence, (\ref{theta}) is true for $A=1$.
\end{proof}

\begin{proof}[\bf\em Proof of Theorem \ref{S12}]
	For $k=1$, $S_n(k)=S_n$ and the result follows. Let $\lceil\sqrt{n}\rceil=p$, we will prove that $\chi_L(S_n(2))=p+1$.
	
	First, we construct a locating $(p+1)$-coloring of $S_n(2)$. Let $A=\{(x,y) \mid x\in \{1,2,\cdots,p\}, y\in\{1,2,\cdots,p+1\}-\{x\}\}$, then $|A|=p^2\geq n$. 
	Define a coloring $c$ with $c(a_0)=p+1$, $c(a_{i,1})=x_i$, and $c(a_{i,2})=y_i$; where $(x_i,y_i)$ is the $i^{th}$ element in $A$ (based on lexicographic order). Note that $c$ is a locating $(p+1)$-coloring of $S_n(2)$, because vertices in the same level is resolved by the color of their neighbors, and vertices in different level is resolved by $c^{-1}(p+1)$.
	
	To prove that $(p+1)$ is minimum, let $c'$ be a $q$-coloring of $S_n(2)$ with $c(a_0)=q$ and $q\leq p$. We will prove that $c'$ is not a locating coloring. Let $A'=\{(c'(a_{i,1}),c'(a_{i,2})) \mid i=1,2\cdots,n\}$. Note that $A'\subseteq \{(x,y) \mid x\in \{1,2,\cdots,p-1\}, y\in\{1,2,\cdots,p\}-\{x\}\}$, so $|A'|\leq (p-1)^2<n$. This means, there are two different indices $i$ and $j$, such that $(c'(a_{i,1}),c'(a_{i,2}))=(c'(a_{j,1}),c'(a_{j,2}))$. Therefore $a_{c'}(a_{i,1})=a_{c'}(a_{j,1})$ and thus $c'$ is not a locating coloring of $S_n(2)$.
\end{proof}

\begin{proof}[\bf\em Proof of Theorem \ref{S3}]
	Note that $f(p)=(p-1)\left\lfloor\frac{p^2}{4}\right\rfloor-\left\lfloor\frac{p^2-2p}{4}\right\rfloor=\left\lceil \frac{p}{2} \right\rceil(p-1)\left\lfloor \frac{p}{2}-1 \right\rfloor+\left\lceil \frac{p}{2} \right\rceil^2$ is a strictly increasing function for $p\geq 3$. So, for $n\geq 2$ there is a unique $p$ such that $f(p-1)< n\leq f(p)$.
	Proving (\ref{Sn3}) is equivalent to proving $\chi_L(S_n(3))= p$ where $p$ is the smallest integer such that $n\leq f(p)$.
	
	Let $n\geq2$ and $\chi_L(S_n(3))=p$.
	First, we will prove that $n\leq f(p)$.
	Let $V(S_n(3))=\{v\}\cup\{x_i,y_i,z_i \mid i=1,2,\cdots,n\}$ and for $i=1,2,\cdots,n$; the subrgaph induced by $\{v,x_i,y_i,z_i\}$ is a path.
	Let $c$ be a locating $p$-coloring of $S_n(3)$ with $c(v)=p$. Let $A=\{j\mid d(v,c^{-1}(j))=1\}$, $B=\{j \mid d(v,c^{-1}(j))=2\}$, and $C=\{j\mid d(v,c^{-1}(j))=3\}$; also $|A|=\alpha$, $|B|=\beta$, and $|C|=\gamma$. 
	
	Now, we will count the number of possible color code for $x_i$. We know that $c(x_i)\in A$ and $c(y_i)\in A\cup B\cup\{p\}$. Let $a_c(x_i)=(a_1,a_2,\cdots,a_p)$, then $a_p=1$, $a_j=2$ for $j\in A\backslash \{c(x_i),c(y_i)\}$, $a_j=3$ for $j\in B\backslash \{c(y_i),c(z_i)\}$, and $a_j=4$ for $j\in C\backslash\{c(z_i)\}$. So, for fix $A$, $B$, and $C$, the color code of $x_i$ depends only on $c(x_i)$, $c(y_i)$, and $c(z_i)$.
	
	\textbf{(i)} If $c(y_i)\in B$ and $c(z_i)\in A\cup \{p\}$, then there are $\alpha$ possible choices for $c(x_i)$ and $\beta$ possible choices for $c(y_i)$. So, there are $\alpha\beta$ possible values for $a_c(x_i)$; note that the value of $a_c(x_i)$ does not change for different values of $c(z_i)\in A\cup \{p\}$.
	\textbf{(ii)} If $c(y_i)\in B$ and $c(z_i)\in (B\cup C)\backslash \{c(y_i)\}$, then there are $\alpha$ possible choices for $c(x_i)$, $\beta$ possible choices for $c(y_i)$, and $\beta+\gamma-1$ possible choices for $c(z_i)$. So, there are $\alpha\beta(\beta+\gamma-1)$ possible values for $a_c(x_i)$.
	\textbf{(iii)} If $c(y_i)\in (A\cup\{p\})\backslash\{c(x_i)\}$ and $c(z_i)\in (A\cup\{p\})\backslash\{c(y_i)\}$, then there are $\alpha$ possible choices for $c(x_i)$ and $\alpha$ possible choices for $c(y_i)$. So, there are $\alpha^2$ possible values for $a_c(x_i)$; note that the value of $a_c(x_i)$ does not change for different values of $c(z_i)\in A\cup \{p\}\backslash\{c(y_i)\}$.
	\textbf{(iv)} If $c(y_i)\in (A\cup\{p\})\backslash\{c(x_i)\}$ and $c(z_i)\in (B\cup C)$, then there are $\alpha$ possible choices of $c(x_i)$, $\alpha$ possible choices for $c(y_i)$, and $\beta+\gamma$ possible values for $c(z_i)$. So, there are $\alpha^2(\beta+\gamma)$ possible values for $a_c(x_i)$.
	
	Since $\alpha+\beta+\gamma=p-1$, the total number of possible values for $a_c(x_i)$ is $\alpha\beta+\alpha\beta(\beta+\gamma-1)+\alpha^2+\alpha^2(\beta+\gamma)=\alpha(\alpha+\beta)(\beta+\gamma)+\alpha^2=\alpha(p-1-\gamma)(p-1-\alpha)+\alpha^2$.
	This value is maximized by taking $\gamma=0$, so the number of possible values for $a_c(x_i)$ is at most $\alpha(p-1)(p-1-\alpha)+\alpha^2$ which is maximized when $\alpha=\frac{p}{2}+\frac{1}{2p-4}$. Since $\alpha$ must be an integer and the closest integer to $\frac{p}{2}+\frac{1}{2p-4}$ is $\left\lceil \frac{p}{2} \right\rceil$, then the number of possible values for $a_c(x_i)$ is at most 
	\[
	\left\lceil \frac{p}{2} \right\rceil(p-1)\left\lfloor \frac{p}{2}-1 \right\rfloor+\left\lceil \frac{p}{2} \right\rceil^2
	=
	(p-1)\left\lceil \frac{p}{2} \right\rceil\left\lfloor \frac{p}{2} \right\rfloor - \left\lceil \frac{p}{2} \right\rceil\left(p-1-\left\lceil \frac{p}{2} \right\rceil\right)
	=
	(p-1)\left\lfloor\frac{p^2}{4}\right\rfloor-\left\lfloor\frac{p^2-2p}{4}\right\rfloor.
	\]
	Therefore, $n\leq f(p)$.
	
	Now, we prove that $p$ is the smallest integer such that $n\leq f(p)$. Suppose otherwise, let $k$ be an integer such that $n\leq f(k)$ and $k<p$. We will construct a locating $k$-coloring of $S_n(3)$, which will contradict $\chi_L(S_n(3))=p$.
	
	Let $A=\left\{1,2,\cdots,\left\lceil \frac{k}{2} \right\rceil\right\}$ and $B=\left\{\left\lceil \frac{k}{2} \right\rceil+1,\left\lceil \frac{k}{2} \right\rceil+2,\cdots,k-1\right\}$. Let $S_1=\{(a,b,a) \mid a\in A; b\in B\}$, $S_2=\{(a,b,c)\mid a\in A; b,c\in B; b\ne c\}$, $S_3=\{(a,b,a)\mid a,b\in A\cup\{k\}; b\ne a\ne k\}$, $S_4=\{(a,b,c)\mid a,b\in A\cup\{k\}; c\in B; b\ne a\ne k\}$, and $S=S_1\cup S_2\cup S_3\cup S_4$. Note that $|S|=\alpha\beta+\alpha\beta(\beta-1)+\alpha^2+\alpha^2\beta$ where $\alpha=|A|$ and $\beta=|B|$, so $|S|=f(k)\geq n$.
	Let $c$ be a coloring of $S_n(3)$ as follows.
	
	\begin{enumerate}
		\item Color $v$ with $k$.
		\item Arrange the elements of $S$ from $S_1$ to $S_4$.
		\item For $i=1,2,\cdots,n$; let $(a_i,b_i,c_i)$ be the $i^{th}$ element in $S$ (based on the previous ordering), and color $x_i$ with $a_i$, color $y_i$ with $b_i$, and color $z_i$ with $c_i$.
	\end{enumerate}
	
	Now we prove that $c$ is a locating coloring. Note that $\{w\in V(S_n(3)) \mid c(w)=k\}\subseteq \{v,y_1,y_2,\cdots,y_n\}$ and $c(z_i)\in A \Rightarrow c(z_i)=c(x_i)$. 		
	By contradiction, let $u,w$ be two vertices with $a_c(u)=a_c(w)$. Consider the following cases.
	
	\textit{Case I : $d(u,c^{-1}(k))$ is even.}
	Since $S_n(3)$ is bipartite and color $k$ only appears in one partition, if $d(u,c^{-1}(k))$ is even, then $u,w\in \{v,y_1,y_2,\cdots,y_n\}$.
	Note that $v$ is the only vertex with $d(v,r)=1$ for all $r\in A$, so $v\notin\{u,w\}$. Let $u=y_i$ and $w=y_j$ with $i\ne j$. Since $a_c(u)=a_c(w)$, then $\{c(x_i),c(z_i)\}=\{c(x_j),c(z_j)\}$. If $c(z_i)=c(z_j)$, then $c(x_i)=c(x_j)$, which implies $i=j$. If $c(z_i)=c(x_j)$, then $c(z_i)\in A$, which implies $|\{c(x_i),c(z_i),c(x_j),c(z_j)\}|=1$ and thus $i=j$.
	
	\textit{Case II : $d(u,c^{-1}(k))$ is odd.}
	Let $u\in \{x_i,y_i,z_i\}$ and $w\in \{x_j,y_j,z_j\}$. If $d(u,v)=d(w,v)$, then $i\ne j$ and the different colors in $(a_i,b_i,c_i)$ and $(a_j,b_j,c_j)$ will distinguish $a_c(u)$ and $a_c(w)$. If $d(u,v)< d(w,v)$, then $u=x_i$ and $w=z_j$. This implies $c(z_j)=c(x_i)\in A$, thus $c(z_j)=c(x_j)$. Since $d(z_j,c^{-1}(k))=d(x_i,c^{-1}(k))=1$, then $c(y_j)=p$ and $d(z_j,r)=5>3\geq d(x_i,r)$ for every $r\in B$, a contradiction.
	
	We already prove that $c$ is a locating $k$ coloring of $S_n(3)$, which means $\chi_L(S_n(3))\leq k$, but $\chi_L(S_n(3))=p>k$, a contradiction. Therefore $p$ is the smallest integer such that $n\leq f(p)$, and thus (\ref{Sn3}) is proven.
	
	From (\ref{Sn3}), we have
	\begin{align*}
	\lim\limits_{n\to \infty} \frac{n}{f(p)}=1\quad	\Rightarrow \quad
	\lim\limits_{n\to \infty} \frac{n}{p^3/4}=1\quad \Rightarrow \quad
	\lim\limits_{n\to \infty} \frac{p}{\sqrt[3]{4n}}=1,
	\end{align*}
	therefore $\chi_L\left(S_n(3)\right) =(1+o(1)) \sqrt[3]{4n}$.
\end{proof}

\section*{Acknowledgment}
This research has been funded by the Indonesian Ministry of Research and Technology/National Agency of Research and Innovation under the World Class University (WCU) Program managed by Institut Teknologi Bandung.

\end{document}